\documentclass[12pt,reqno]{amsart}
\usepackage{amsmath, amsfonts, amssymb, amsthm, hyperref}
\usepackage{bm}
\def\l{\left}
\def\r{\right}
\def\bg{\bigg}
\def\({\bg(}
\def\){\bg)}
\def\t{\text}
\def\f{\frac}

\def\bi{\binom}

\def\eq{\equiv}

\def\Z{\mathbb Z}

\def\N{\mathbb N}

\def\R{\mathbb R}

\def\[{\begin{equation}}
\def\]{\end{equation}}
\def\<{\langle}
\def\>{\rangle}
\def\1{{\bf 1}}

\theoremstyle{plain}
\newtheorem{theorem}{Theorem}
\newtheorem{conjecture}{Conjecture}
\newtheorem{proposition}{Proposition}
\newtheorem{lemma}{Lemma}

\theoremstyle{definition}

\theoremstyle{remark}

\numberwithin{equation}{section}

\begin{document}
\title{Supercongruences concerning truncated hypergeometric series}
\author{Chen Wang}
\address {Department of Mathematics, Nanjing
University, Nanjing 210093, People's Republic of China}
\email{cwang@smail.nju.edu.cn}
\author{Hao Pan}
\address {School of Applied Mathematics, Nanjing University of Finance and Economics, Nanjing 210023, People's Republic of China}
\email{haopan79@zoho.com}

\begin{abstract}
Let $n\geq 3$ be an integer and $p$ be a prime with $p\equiv 1\pmod{n}$. In this paper, we show that
$$
_nF_{n-1}\bigg[\begin{matrix} \f{n-1}{n}&\f{n-1}{n}&\ldots&\f{n-1}{n}\\ &1&\ldots&1\end{matrix}\bigg | \, 1\bigg]_{p-1}\eq -\Gamma_p\bigg(\f{1}{n}\bigg)^n\pmod{p^3},
$$
where the truncated hypergeometric series
$$
_nF_{n-1}\bigg[\begin{matrix} x_1&x_2&\ldots&x_n\\ &y_1&\cdots&y_{n-1}\end{matrix}\bigg | \, z\bigg]_m=\sum_{k=0}^{m}\frac{z^k}{k!}\prod_{j=0}^{k-1}\frac{(x_1+j)\cdots(x_{n}+j)}{(y_1+j)\cdots(y_{n-1}+j)}
$$
and $\Gamma_p$ denotes the $p$-adic gamma function.
This confirms a conjecture of Deines, Fuselier, Long, Swisher and Tu in \cite{DFLST16}.
Furthermore, under the same assumptions, we also prove that
$$
p^n\cdot {}_{n+1}F_{n}\bigg[\begin{matrix} 1&1&\ldots&1\\
&\f{n+1}{n}&\ldots&\f{n+1}{n}\end{matrix}\bigg | \, 1\bigg]_{p-1}
\eq-\Gamma_p\l(\f{1}{n}\r)^n\pmod{p^3},
$$
which solves another conjecture in 
\cite{DFLST16}.
\end{abstract}

\thanks{The first author is supported by the National Natural Science Foundation of China (Grant No. 11571162).}

\maketitle

\section{introduction}
\setcounter{lemma}{0} \setcounter{theorem}{0}
\setcounter{equation}{0}\setcounter{proposition}{0}

Define the truncated hypergeometric series
$$
{}_{n}F_{n-1}\bigg[\begin{matrix}x_1&x_2&\ldots&x_{n}\\
&y_1&\ldots&y_{n-1}\end{matrix}\bigg|\,z\bigg]_{m}:=\sum_{k=0}^m\frac{(x_1)_k(x_2)_k\cdots(x_n)_k}{(y_1)_k\cdots (y_{n-1})_k}\cdot\frac{z^k}{k!},
$$
where
$$
(x)_k=\begin{cases}
x(x+1)\cdots(x+k-1),&\text{if }k\geq 1,\\
1,&\text{if }k=0.
\end{cases}
$$
Clearly, the truncated hypergeometric series is a finite analogue of the original hypergeometric series
$$
{}_{n}F_{n-1}\bigg[\begin{matrix}x_1&x_2&\ldots&x_{n}\\
&y_1&\ldots&y_{n-1}\end{matrix}\bigg|\,z\bigg]:=\sum_{k=0}^\infty\frac{(x_1)_k(x_2)_k\cdots(x_n)_k}{(y_1)_k\cdots (y_{n-1})_k}\cdot\frac{z^k}{k!},
$$
In the recent years, many insteresting $p$-adic supercongruences concerning the truncated hypergeometric series have been established (cf. \cite{AhOn00,DFLST16,He17,Liu17,LoRa16,MS15,Mo04,Mo05,OsSc09,OsStZu,SunZH13,SunZW11,SunZW13,Ta12}).
Especially, the truncated hypergeometric series have a very close connection with the $p$-adic Gamma function. Suppose that $p$ is an odd prime. Let $\Z_p$ denote the ring of all $p$-adic integers and let $|\cdot|_p$ denote the $p$-adic norm over $\Z_p$. For each integer $n\geq 1$, define the $p$-adic Gamma function
$$
\Gamma_p(n):=(-1)^n\prod_{\substack{1\leq k<n\\ (k,p)=1}}k.
$$
In particular, set $\Gamma_p(0)=1$. And for any $x\in\Z_p$, define
$$
\Gamma_p(x):=\lim_{\substack{n\in\N\\ |x-n|_p\to 0}}\Gamma_p(n).
$$

In \cite{DFLST16}, Deines, Fuselier, Long, Swisher and Tu proposed the following beautiful conjecture.
\begin{conjecture}\label{DFLSTconj}
Suppose that $n\geq 3$ is an integer and $p$ is a prime with $p\equiv 1\pmod{n}$. Then
\begin{equation}\label{nFn1Gammap}
_nF_{n-1}\bigg[\begin{matrix} \f{n-1}{n}&\f{n-1}{n}&\ldots&\f{n-1}{n}\\ &1&\ldots&1\end{matrix}\bigg | \, 1\bigg]_{p-1}\eq -\Gamma_p\bigg(\f{1}{n}\bigg)^n\pmod{p^3}.
\end{equation}
\end{conjecture}
In the same paper, Deines, Fuselier, Long, Swisher and Tu had shown that both sides of (\ref{nFn1Gammap}) are congruent modulo $p^2$.

In this paper, we shall confirm the above conjectures of Deines, Fuselier, Long, Swisher and Tu.
\begin{theorem}\label{main}
Conjecture \ref{DFLSTconj} is true.
\end{theorem}
Furthermore, 
Deines, Fuselier, Long, Swisher and Tu \cite[Eq. (4) of Section 7]{DFLST16} also conjectured that
\begin{equation}\label{nFn1Gammap2}
p^n\cdot {}_{n+1}F_{n}\bigg[\begin{matrix} 1&1&\ldots&1\\
&\f{n+1}{n}&\ldots&\f{n+1}{n}\end{matrix}\bigg | \, 1\bigg]_{p-1}
\eq-\Gamma_p\l(\f{1}{n}\r)^n\pmod{p^3},
\end{equation}
where $n\geq 3$ is an integer and $p$ is a prime with $p\equiv 1\pmod{n}$. 
\begin{theorem}\label{main2}
(\ref{nFn1Gammap2}) is true.
\end{theorem}
In the next section, we shall establish an auxiliary proposition. In the third section, the proof of Theorem \ref{main} will  be given. Finally, Theorem \ref{main2} will be proved in the last section.

\section{An auxiliary result}
\setcounter{lemma}{0} \setcounter{theorem}{0}
\setcounter{equation}{0}\setcounter{proposition}{0}

In this section, we shall prove the following proposition, which is the key ingredient of our proof of Theorem \ref{main}.
\begin{proposition}\label{1n1n1jn12} Let $n\geq3$ be an integer and suppose that $p\equiv 1\pmod{n}$ is prime. Then
\begin{equation}\label{1n1n1jn12e}
\sum_{k=0}^{p-1}\frac{(1-n^{-1})_k^n}{(1)_k^n}\sum_{j=1}^k\frac{1}{(j-n^{-1})^2}\equiv\Gamma_p\bigg(\frac1n\bigg)^n\cdot\bigg(G_2\bigg(\frac1n\bigg)-G_1\bigg(\frac1n\bigg)^2\bigg)\pmod{p},
\end{equation}
where
$$
G_1(x):=\frac{\Gamma_p'(x)}{\Gamma_p(x)},\qquad
G_2(x):=\frac{\Gamma_p''(x)}{\Gamma_p(x)}.
$$
\end{proposition}
To prove Proposition \ref{1n1n1jn12}, we need  several lemmas.
The first one is a consequence of the classical Karlsson-Minton identity \cite{Ka71}.
\begin{lemma}\label{KM}
For any nonnegative integers $m_1, \cdots, m_n$,
\begin{align*}
&_{n+1}F_{n}\bigg[\begin{matrix} -(m_1+\cdots+m_2)&b_1+m_1&\ldots&b_n+m_n\\ &b_1&\ldots&b_n\end{matrix}\bigg | \, 1\bigg]
\\=&(-1)^{m_1+\cdots+m_n}\cdot\f{(m_1+\cdots+m_n)!}{(b_1)_{m_1}\cdots(b_n)_{m_n}}.
\end{align*}
\end{lemma}
If $p$ is prime and $q$ is a power of $p$, then for any $x\in\Z_p$, let $\langle x\rangle_{q}$ denote the least non-negative residue of $x$ modulo $q$. The following lemma determines $G_1(x)$ modulo $p^2$ and $G_2(x)$ modulo $p$.
\begin{lemma}\label{G1xG2x}
Suppose that $x\in\Z_p$. Then
\begin{equation}
\label{con1}G_1(x)\eq G_1(0)+\sum_{\substack{1\leq j<\langle x\rangle_{p^2}\\(j,p)=1}}\f{1}{j}\pmod{p^2},
\end{equation}
\begin{equation}
\label{con2}G_2(x)\eq G_2(0)+2G_1(0)\sum_{1\leq j<\langle x\rangle_p}\f{1}{j}+2\sum_{1\leq i<j<\langle x\rangle_p}\f{1}{ij}\pmod{p}.
\end{equation}
\end{lemma}
\begin{proof} (\ref{con1}) and (\ref{con2}) should not be new, though we can't find them in any literature. However, for the sake of completeness, here we give the proof of (\ref{con1}) and (\ref{con2}).
First, choose a sufficiently large integer $h$ such that
$$\Gamma_p'(x)\eq\f{\Gamma_p(x+p^h)-\Gamma_p(x)}{p^h}\pmod{p^2}\ \ \t{and}\ \ \Gamma_p'(0)\eq \f{\Gamma_p(p^h)-1}{p^h}\pmod{p^2}.$$
Let $a=\langle x\rangle_{p^{h+2}}.$ Then
\begin{align*}
G_1(x)=\f{\Gamma_p'(x)}{\Gamma_p(x)}\equiv& \f{1}{\Gamma_p(a)}\cdot\f{\Gamma_p(a+p^h)-\Gamma_p(a)}{p^h}=\f{1}{p^h}\cdot\bigg(\Gamma_p(p^h)\prod_{\substack{1\leq j<a\\(j,p)=1}}\f{j+p^h}{j}-1\bigg)
\\\eq&\f{\Gamma_p(p^h)-1}{p^h}+\sum_{\substack{1\leq j<a}}\f{1}{j}\eq \Gamma_p'(0)+\sum_{1\leq j<\langle x\rangle_{p^2}}\f{1}{j}\pmod{p^2}.
\end{align*}
Since $\Gamma_p(0)=1$, we get (\ref{con1}). Next, $b=\langle x\rangle_{p^3}$. Then
\begin{align}\label{Gammapxpx1}
\Gamma_p(x+p)-\Gamma_p(x)\eq&\Gamma_p(b+p)-\Gamma_p(b)=\Gamma_p(b)\bigg(\Gamma_p(p)\prod_{\substack{1\leq j<b\\ (j,p)=1}}\f{p+j}{j}-1\bigg)\notag
\\\eq&\Gamma_p(b)\bigg(\Gamma_p(p)\bigg(1+p\sum_{\substack{1\leq j<b\\(j,p)=1}}\f{1}{j}+p^2\sum_{\substack{1\leq i<j<b\\(ij,p)=1}}\f{1}{ij}\bigg)-1\bigg)\notag
\\\eq&\Gamma_p(x)\bigg((\Gamma_p(p)-1)+p\Gamma_p(p)\sum_{\substack{1\leq j<\langle x\rangle_{p^2}\\(j,p)=1}}\f{1}{j}+p^2\sum_{1\leq i<j<\langle x\rangle_p}\f{1}{ij}\bigg)\pmod{p^3}.
\end{align}
On the other hand, we have
\begin{equation}\label{Gammapxpx2}
\Gamma_p(x+p)-\Gamma_p(x)\eq\Gamma_p'(x)p+\f{\Gamma_p''(x)p^2}{2}\pmod{p^3}
\end{equation}
Combining (\ref{Gammapxpx1}) and (\ref{Gammapxpx2}),
we obtain
\begin{align}\label{GammapdxGammapd2x}
G_1(x)+\frac{G_2(x)p}{2}=\frac{\Gamma_p'(x)}{\Gamma_p(x)}+\f{\Gamma_p''(x)p}{2\Gamma_p(x)}\equiv
\frac{\Gamma_p(p)-1}{p}+\Gamma_p(p)\sum_{\substack{1\leq j<\langle x\rangle_{p^2}\\(j,p)=1}}\f{1}{j}+p\sum_{1\leq i<j<\langle x\rangle_p}\f{1}{ij}\pmod{p^2}
\end{align}
Note that
$$
\Gamma_p(p)\eq 1+\Gamma_p'(0)p+\f{\Gamma_p''(0)p^2}{2}\pmod{p^3}.
$$
So (\ref{con2}) is concluded by substituting (\ref{con1}) into (\ref{GammapdxGammapd2x}).
\end{proof}

\begin{lemma}\label{G10G20}
$G_1(0)^2=G_2(0).$
\end{lemma}
\begin{proof}
By \cite[p.379]{R} we know that
$$
(\log_p\Gamma_p)''(x)=\int_{\Z_p^{\times}}\f{1}{x+t}=-\sum_{m\geq1}x^{2m-1}\int_{\Z_p^{\times}}t^{-2m}dt.
$$
So we have
\begin{align*}
G_2(0)-G_1(0)^2=(\log_p\Gamma_p)''(0)=0.
\end{align*}
\end{proof}

\begin{lemma}\label{1mkn1knj2}
Suppose that $n\geq3$ and $p>3$ be a prime with $p\eq1\pmod{n}$. Then we have
\begin{equation}\label{harcon1}
\sum_{k=0}^{p-1}\f{(1+m)_k^n}{(1)_k^n}\sum_{j=1}^k\f{1}{j}\eq\sum_{k=0}^{p-1}\f{(1+m)_k^n}{(1)_k^n}\sum_{j=1}^{k+m}\f{1}{j}\pmod{p}
\end{equation}
and
\begin{equation}\label{harcon2}
\sum_{k=0}^{p-1}\f{(1+m)_k^n}{(1)_k^n}\l(\sum_{j=1}^k\f{1}{j}\r)^2\eq\sum_{k=0}^{p-1}\f{(1+m)_k^n}{(1)_k^n}\l(\sum_{j=1}^{k+m}\f{1}{j}\r)^2\pmod{p}
\end{equation}
and
\begin{equation}\label{con3}
\sum_{k=0}^{p-1}\f{(1+m)_k^n}{(1)_k^n}\sum_{j=1}^k\f{1}{j^2}\eq-\sum_{k=0}^{p-1}\f{(1+m)_k^n}{(1)_k^n}\sum_{j=1}^{k+m}\f{1}{j^2}\pmod{p},
\end{equation}
where $m=(p-1)/n$.
\end{lemma}
\begin{proof}
We only give the proof of \eqref{con3}, \eqref{harcon1} and \eqref{harcon2} can be derived in a similar way.
Clearly
$$
\f{(1+m)_k}{(1)_k}=\bi{m+k}{m}\eq0\pmod{p}$$ for each $p-m\leq k\leq p-1$. So
$$
\sum_{k=0}^{p-1}\f{(1+m)_k^n}{(1)_k^n}\sum_{j=1}^k\f{1}{j^2}\eq\sum_{k=0}^{p-1-m}\bi{m+k}{m}^n\sum_{j=1}^k\f{1}{j^2}\pmod{p}.
$$
And
\begin{align*}
&\sum_{k=0}^{p-1-m}\bi{m+k}{m}^n\sum_{i=1}^k\f{1}{j^2}=\sum_{k=m}^{p-1}\bi{m+p-1-k}{m}^n\sum_{j=1}^{p-1-k}\f{1}{j^2}\\
\eq&-\sum_{k=m}^{p-1}\bi{m+p-1-k}{m}^n\sum_{j=1}^{k}\f{1}{j^2}
=-\sum_{k=0}^{p-1-m}\bi{p-1-k}{m}^n\sum_{j=1}^{k+m}\f{1}{j^2}\\
\equiv&-\sum_{k=0}^{p-1-m}\bi{-1-k}{m}^n\sum_{j=1}^{k+m}\f{1}{j^2}\pmod{p}.
\end{align*}
Since $mn=p-1$ is even and $n\geq3$, we get
\begin{align*}
\sum_{k=0}^{p-1}\f{(1+m)_k^n}{(1)_k^n}\equiv-\sum_{k=0}^{p-1-m}(-1)^{mn}\bi{m+k}{m}^n\sum_{j=1}^{k+m}\f{1}{j^2}\equiv-\sum_{k=0}^{p-1}\bi{m+k}{m}^n\sum_{j=1}^{k+m}\f{1}{j^2}\pmod{p}.
\end{align*}
\end{proof}
\begin{lemma} For any integer $k\geq 0$ and $\alpha,\beta\in\R$,
\begin{equation}\label{akd1}
\frac1{(1+\alpha+x)_k}\frac{d}{d x}(1+\alpha+x)_k=\sum_{j=1}^{k}\frac{1}{j+\alpha+x},
\end{equation}
\begin{equation}\label{akd2}
\frac1{(1+\alpha+x)_k}\frac{d^2}{d x^2}(1+\alpha+x)_k=\bigg(\sum_{j=1}^{k}\frac{1}{j+\alpha+x}\bigg)^2-\sum_{j=1}^{k}\frac{1}{(j+\alpha+x)^2},
\end{equation}
\begin{equation}\label{bkd1}
(1+\beta+x)_k\frac{d}{d x}\bigg(\frac1{(1+\beta+x)_k}\bigg)=-\sum_{j=1}^{k}\frac{1}{j+\beta+x},
\end{equation}
\begin{equation}\label{bkd2}
(1+\beta+x)_k\frac{d^2}{d x^2}\bigg(\frac1{(1+\beta+x)_k}\bigg)=\bigg(\sum_{j=1}^{k}\frac{1}{j+\beta+x}\bigg)^2+\sum_{j=1}^{k}\frac{1}{(j+\beta+x)^2},
\end{equation}
\begin{equation}\label{abkd1}
\frac{(1+\beta+x)_k}{(1+\alpha+x)_k}\frac{d}{d x}\bigg(\frac{(1+\alpha+x)_k}{(1+\beta+x)_k}\bigg)=\sum_{j=1}^{k}\frac{1}{j+\alpha+x}-\sum_{j=1}^{k}\frac{1}{j+\beta+x},
\end{equation}
\begin{align}\label{abkd2}
&\frac{(1+\beta+x)_k}{(1+\alpha+x)_k}\frac{d^2}{d x^2}\bigg(\frac{(1+\alpha+x)_k}{(1+\beta+x)_k}\bigg)\notag\\
=&\bigg(\sum_{j=1}^{k}\frac{1}{j+\alpha+x}-\sum_{j=1}^{k}\frac{1}{j+\beta+x}\bigg)^2+\sum_{j=1}^{k}\frac{1}{(j+\beta+x)^2}-\sum_{j=1}^{k}\frac{1}{(j+\alpha+x)^2}
\end{align}
\begin{proof}
Those identities can be verified directly.
\end{proof}
\end{lemma}
Now we are ready to prove Proposition \ref{1n1n1jn12}.
\begin{proof}[Proof of Proposition \ref{1n1n1jn12}]
Let $m=(p-1)/n$.
According to Lemmas \ref{G1xG2x} and \ref{G10G20}, we have
$$
G_2\bigg(\frac1n\bigg)-G_1\bigg(\frac1n\bigg)^2\equiv-
\sum_{k=1}^{\langle n^{-1}\rangle_p-1}\frac1{k^2}\equiv
\sum_{k=1}^{m}\frac1{k^2}\pmod{p}.
$$
And since
\begin{equation}\label{Gammapx1x}
\Gamma_p(x)\Gamma_p(1-x)=(-1)^{\langle-x\rangle_p-1}
\end{equation}
for any $x\in\Z_p$,
$$
\frac{(p-1)!}{(1)_m^n}\equiv-\frac{(-1)^{(m+1)n}}{\Gamma_{p}(m+1)^n}=-\Gamma_p(-m)^n\equiv -\Gamma_p\bigg(\frac1n\bigg)^n\pmod{p}.
$$
Thus it suffices to show that
\begin{equation}\label{1mkn1knjm2}
\sum_{k=0}^{p-1}\frac{(1-p)_k(1+m)_k^n}{(1)_k^{n+1}}\sum_{j=1}^k\frac{1}{(j+m)^2}\equiv-\frac{(p-1)!}{(1)_m^n}\sum_{k=1}^m\frac1{k^2}\pmod{p}.
\end{equation}

Set
\begin{align*}
\Psi(x,y):=_{n+1}F_n\bigg[\begin{matrix} 1-p&1+m+x&1+m+y&1+m&\cdots&1+m\\ &1+x&1+y&1&\cdots&1\end{matrix}\bigg | \, 1\bigg]
\end{align*}
It follows from (\ref{abkd1}) and (\ref{abkd2}) that
$$
\frac{\partial^2\Psi(x,y)}{\partial x\partial y}\bigg|_{(x,y)=(0,0)}=\sum_{k=0}^{p-1}\frac{(1-p)_k(1+m)_k^n}{(1)_k^{n+1}}\bigg(\sum_{j=1}^k\frac{1}{j+m}-\sum_{j=1}^k\frac1j\bigg)^2
$$
and
\[\label{lastsection1}
\frac{\partial^2\Psi(x,y)}{\partial x^2}\bigg|_{(x,y)=(0,0)}=\sum_{k=0}^{p-1}\frac{(1-p)_k(1+m)_k^n}{(1)_k^{n+1}}\bigg(\bigg(\sum_{j=1}^k\frac{1}{j+m}-\sum_{j=1}^k\frac1j\bigg)^2+\sum_{j=1}^k\frac1{j^2}-\sum_{j=1}^k\frac1{(j+m)^2}\bigg).
\]
Hence
\begin{equation}\label{Psixyx2A}
\frac{\partial^2\Psi(x,y)}{\partial x\partial y}\bigg|_{(x,y)=(0,0)}-\frac{\partial^2\Psi(x,y)}{\partial x^2}\bigg|_{(x,y)=(0,0)}=\sum_{k=0}^{p-1}\frac{(1-p)_k(1+m)_k^n}{(1)_k^{n+1}}\bigg(\sum_{j=1}^k\frac{1}{(j+m)^2}-\sum_{j=1}^k\frac1{j^2}\bigg).
\end{equation}

On the other hand, with help of Lemma \ref{KM}, we have
\begin{equation}\label{PsixyKM}
\Psi(x,y)=\f{(p-1)!}{(1+x)_m(1+y)_m(1)_m^{n-2}}.
\end{equation}
Applying (\ref{bkd1}) and (\ref{bkd2}), we get
$$
\frac{\partial^2\Psi(x,y)}{\partial x\partial y}\bigg|_{(x,y)=(0,0)}=
\f{(p-1)!}{(1)_m^{n}}\bigg(\sum_{j=1}^m\frac1j\bigg)^2
$$
and
\[\label{lastsection2}
\frac{\partial^2\Psi(x,y)}{\partial x^2}\bigg|_{(x,y)=(0,0)}=\f{(p-1)!}{(1)_m^{n}}\bigg(\bigg(\sum_{j=1}^m\frac1j\bigg)^2+\sum_{j=1}^m\frac1{j^2}\bigg).
\]
So
\begin{equation}\label{Psixyx2B}
\frac{\partial^2\Psi(x,y)}{\partial x\partial y}\bigg|_{(x,y)=(0,0)}-\frac{\partial^2\Psi(x,y)}{\partial x^2}\bigg|_{(x,y)=(0,0)}=-\f{(p-1)!}{(1)_m^{n}}\sum_{j=1}^m\frac1{j^2}.
\end{equation}

Combining (\ref{Psixyx2A}) and (\ref{Psixyx2B}), we obtain that
\begin{equation}\label{1mkn1knjm2j2}
\sum_{k=0}^{p-1}\frac{(1-p)_k(1+m)_k^n}{(1)_k^{n+1}}\bigg(\sum_{j=1}^k\frac{1}{(j+m)^2}-\sum_{j=1}^k\frac1{j^2}\bigg)\equiv-\f{(p-1)!}{(1)_m^{n}}\sum_{j=1}^m\frac1{j^2}\pmod{p}.
\end{equation}
Furthermore, by (\ref{PsixyKM}),
$$
\sum_{k=0}^{p-1}\frac{(1-p)_k(1+m)_k^n}{(1)_k^{n+1}}=\Psi(0,0)=\f{(p-1)!}{(1)_m^{n}}.
$$
Therefore
\begin{align*}
0\equiv&\sum_{k=0}^{p-1}\frac{(1-p)_k(1+m)_k^n}{(1)_k^{n+1}}\bigg(\sum_{j=1}^k\frac{1}{(j+m)^2}-\sum_{j=1}^k\frac1{j^2}+\sum_{j=1}^m\frac1{j^2}\bigg)\\
\equiv&
\sum_{k=0}^{p-1}\frac{(1+m)_k^n}{(1)_k^{n}}\bigg(\sum_{j=1}^{k+m}\frac{1}{j^2}-\sum_{j=1}^k\frac1{j^2}\bigg)
\pmod{p}.
\end{align*}

However, in view of Lemma \ref{1mkn1knj2},
$$
\sum_{k=0}^{p-1}\frac{(1+m)_k^n}{(1)_k^{n}}\sum_{j=1}^k\frac1{j^2}\equiv-\sum_{k=0}^{p-1}\frac{(1+m)_k^n}{(1)_k^{n}}\sum_{j=1}^{k+m}\frac{1}{j^2}\pmod{p}.
$$
So we must have
\begin{equation}\label{1mkn1knj20}
\sum_{k=0}^{p-1}\frac{(1+m)_k^n}{(1)_k^{n}}\sum_{j=1}^k\frac1{j^2}\equiv 0\pmod{p}
\end{equation}
Substituting (\ref{1mkn1knj20}) into (\ref{1mkn1knjm2j2}), we get (\ref{1mkn1knjm2}).

\end{proof}

\section{Proof of Theorem \ref{main}}
\setcounter{lemma}{0} \setcounter{theorem}{0}
\setcounter{equation}{0}\setcounter{proposition}{0}
In this section, we shall complete the proof of Theorem \ref{main}.
Let $$m=\f{p-1}{n}.$$
Set
$$
\Phi(x,y):=_{n}F_{n-1}\bigg[\begin{matrix} 1+m-x&1+m-y&\cdots&1+m-y\\ &1&\cdots&1\end{matrix}\bigg | \, 1\bigg]_{p-1}
$$
and
$$
\Omega(x):=_{n}F_{n-1}\bigg[\begin{matrix} 1+m-x&1+m-x&\cdots&1+m-x\\ &1&\cdots&1\end{matrix}\bigg | \, 1\bigg]_{p-1}.
$$
Clearly
$$
\frac{d\Omega(x)}{d x}=n\frac{\partial\Phi(x,y)}{\partial x}\bigg|_{y=x}.
$$
And it is easy to check that
$$
\frac{d^2\Omega(x)}{d x^2}=n\frac{\partial^2\Phi(x,y)}{\partial x\partial y}\bigg|_{y=x}+n\frac{\partial^2\Phi(x,y)}{\partial^2 x}\bigg|_{y=x}.
$$
Hence \begin{align}
&{}_{n}F_{n-1}\bigg[\begin{matrix} 1-n^{-1}&1-n^{-1}&\cdots&1-n^{-1}\\ &1&\cdots&1\end{matrix}\bigg | \, 1\bigg]_{p-1}\notag\\
=&\Omega\bigg(\frac{p}n\bigg)
\equiv\Omega(0)+\frac{\Omega'(0)}{n}\cdot p+\frac{\Omega''(0)}{2n^2}\cdot p^2
\notag\\
=&\Phi(0,0)+\Phi_x'(0,0)p+\frac{\Phi_{xy}''(0,0)+\Phi_{xx}''(0,0)}{2n}\cdot p^2\notag\\
\equiv&\Phi(p,0)+\frac{1}{2n}\cdot\big(\Phi_{xy}''(0,0)-(n-1)\Phi_{xx}''(0,0)\big)\cdot p^2\pmod{p^3}.
\end{align}

On the other hand, with help of Lemma \ref{KM}, we have
\begin{align}\label{eq1}
\Phi(p,0)=&_{n}F_{n-1}\bigg[\begin{matrix} 1+m-p&1+m&\cdots&1+m\\ &1&\cdots&1\end{matrix}\bigg | \, 1\bigg]\notag\\
=&
_{n+1}F_n\bigg[\begin{matrix} 1-p&1+m-p&1+m&\cdots&1+m\\ &1-p&1&\cdots&1\end{matrix}\bigg | \, 1\bigg]=\frac{(p-1)!}{(1-p)_m(m!)^{n-1}}.
\end{align}
So in view of (\ref{akd1}) and (\ref{akd2}), we get
\begin{align}\label{nFn1p3}
&{}_{n}F_{n-1}\bigg[\begin{matrix} 1-n^{-1}&1-n^{-1}&\cdots&1-n^{-1}\\ &1&\cdots&1\end{matrix}\bigg | \, 1\bigg]\notag\\
\equiv&\frac{(p-1)!}{(1-p)_m(m!)^{n-1}}+\frac{n-1}{2n}\cdot p^2\sum_{k=0}^{p-1}\frac{(1+m)^n}{(1)_k^n}\sum_{j=1}^k\frac1{(j+m)^2}\pmod{p^3}.
\end{align}
Furthermore, in view of (\ref{Gammapx1x}),
\begin{align}\label{p1f1pmmfn1}
&\frac{(p-1)!}{(1-p)_m(m!)^{n-1}}=(-1)^{(m+1)n}\cdot\frac{\Gamma_p(p)\Gamma_p(1-p)}{\Gamma_p(1-p+m)\Gamma_p(m+1)^{n-1}}
=-\Gamma_p(-m)^{n-1}\Gamma_p(p-m)\notag\\
\equiv&-\bigg(\Gamma_p\bigg(\frac1n\bigg)-\Gamma_p'\bigg(\frac1n\bigg)\frac pn+\Gamma_p''\bigg(\frac1n\bigg)\frac{p^2}{2n^2}\bigg)^{n-1}\bigg(\Gamma_p\bigg(\frac1n\bigg)+\Gamma_p'\bigg(\frac1n\bigg)\frac{(n-1)p}n+\Gamma_p''\bigg(\frac1n\bigg)\frac{(n-1)^2p^2}{2n^2}\bigg)\notag\\
\equiv&-\Gamma_p\bigg(\frac1n\bigg)^n\bigg(1+\frac{n-1}{2n}\cdot p^2G_2\bigg(\frac1n\bigg)-\frac{n-1}{2n}\cdot p^2G_1\bigg(\frac1n\bigg)^2\bigg)\pmod{p^3}.
\end{align}
Combining (\ref{nFn1p3}), (\ref{p1f1pmmfn1}) and Proposition \ref{1n1n1jn12}, Theorem \ref{main} is concluded.\qed

\section{Proof of Theorem \ref{main2}}
In this section, we shall prove that
\begin{equation}\label{modp3}
\sum_{k=0}^{p-1}\l(p\cdot\f{k!}{(1+n^{-1})_k}\r)^n\eq-\Gamma_p\l(\f{1}{n}\r)^n\pmod{p^3}.
\end{equation}
It is easy to check that
$$
\f{(1)_k}{(1+n^{-1})_k}=\f{(1)_{p-1}}{(1-p)_{p-1-k}}\cdot\f{(1-n^{-1}-p)_{p-1-k}}{(1+n^{-1})_{p-1}},
$$
so we deduce that
$$
\sum_{k=0}^{p-1}\l(p\f{k!}{(1+n^{-1})_k}\r)^n= p^n\f{(1)_{p-1}^n}{(1+n^{-1})_{p-1}^n}\sum_{k=0}^{p-1}\f{(1-n^{-1}-p)_k^n}{(1-p)_k^n}.
$$
Now
\begin{align*}
&p^n\f{(1)_{p-1}^n}{(1+n^{-1})_{p-1}^n}=\f{(-n)^n\Gamma_p(1+n^{-1})^n}{\Gamma_p(p+n^{-1})^n\Gamma_p(-p)^n}\\
\eq&\f{1}{\l((1+pG_1(n^{-1})+p^2G_2(n^{-1})/2)\r)^n\cdot\l(1-pG_1(0)+p^2G_2(0)/2\r)^n}.
\end{align*}
Note that
$$
G_1\l(\f{1}{n}\r)\eq G_1(0)+\sum_{\substack{1\leq j<\langle n^{-1}\rangle_{p^2}\\(j,p)=1}}\f{1}{j}\pmod{p^2}
$$
and
$$
G_2\l(\f{1}{n}\r)\eq G_2(0)+2G_1(0)\sum_{1\leq j<\langle n^{-1}\rangle_p}\f{1}{j}+2\sum_{1\leq i<j<\langle n^{-1}\rangle_p}\f{1}{ij}\pmod{p}
$$
by Lemma \ref{G1xG2x}, then we obtain that
\begin{equation}\label{firstpart}
p^n\f{(1)_{p-1}^n}{(1+n^{-1})_{p-1}^n}\eq 1-pn\sum_{\substack{1\leq j<\langle n^{-1}\rangle_{p^2}\\(j,p)=1}}\f{1}{j}+\f{(pn)^2}{2}\l(\sum_{j=1}^m\f{1}{j}\r)^2-\f{p^2n}{2}\sum_{j=1}^m\f{1}{j^2}\pmod{p^3}.
\end{equation}
Now we set
$$
\Upsilon(x)=\sum_{k=0}^{p-1}\f{(1-n^{-1}-x)_k^n}{(1-x)_k^n}.
$$
In view of \eqref{abkd1} and \eqref{abkd2} we obtain that
$$
\Upsilon'(0)=n\sum_{k=0}^{p-1}\f{(1-n^{-1})_k^n}{(1)_k^n}\l(\sum_{i=1}^k\f{1}{i}-\sum_{i=1}^k\f{1}{i-n^{-1}}\r)
$$
and
\begin{align*}
\Upsilon''(0)=&n^2\sum_{k=0}^{p-1}\f{(1-n^{-1})_k^n}{(1)_k^n}\l(\sum_{i=1}^k\f{1}{i}-\sum_{i=1}^k\f{1}{i-n^{-1}}\r)^2\\
&+n\sum_{k=0}^{p-1}\f{(1-n^{-1})_k^n}{(1)_k^n}\l(\sum_{i=1}^k\f{1}{i^2}-\sum_{i=1}^k\f{1}{(i-n^{-1})^2}\r).
\end{align*}
By Theorem \ref{main} and \eqref{harcon1} we immediately arrive at
$$
\Upsilon''(0)\eq-\Gamma_p\l(\f{1}{n}\r)^n\l(\f{(pn)^2}{2}\l(\sum_{i=1}^m\f{1}{i}\r)^2+\f{p^2n}{2}\sum_{i=1}^m\f{1}{i^2}\r)\pmod{p}.
$$
Furthermore, we have
\begin{equation}\label{secondpart}
\begin{aligned}
&\Upsilon(p)\eq\Upsilon(0)+p\Upsilon'(0)+\f{p^2}{2}\Upsilon''(0)\\
\eq&-\Gamma_p\l(\f{1}{n}\r)^n\left(1-\Gamma_p\l(\f{1}{n}\r)^{-n}pn\sum_{k=0}^{p-1}\f{(1-n^{-1})_k^n}{(1)_k^n}\l(\sum_{i=1}^k\f{1}{i}-\sum_{i=1}^k\f{1}{i-n^{-1}}\r)\right.\\
&\left.+\f{(pn)^2}{2}\l(\sum_{i=1}^m\f{1}{i}\r)^2+\f{p^2n}{2}\sum_{i=1}^m\f{1}{i^2}\right)\pmod{p^3}.
\end{aligned}
\end{equation}
Combining \eqref{firstpart} and \eqref{secondpart} we deduce that
\begin{align*}
\sum_{k=0}^{p-1}\l(p\f{k!}{(1+n^{-1})_k}\r)^n\eq&-\Gamma_p\l(\f{1}{n}\r)^n\Bigg(1-pn\sum_{\substack{1\leq j<\langle n^{-1}\rangle_{p^2}\\(j,p)=1}}\f{1}{j}\\
&-\Gamma_p\l(\f{1}{n}\r)^{-n}pn\sum_{k=0}^{p-1}\f{(1-n^{-1})_k^n}{(1)_k^n}\l(\sum_{i=1}^k\f{1}{i}-\sum_{i=1}^k\f{1}{i-n^{-1}}\r)\Bigg)\pmod{p^3}.
\end{align*}
Then it suffices to show that
\begin{equation}\label{modp2}
\sum_{k=0}^{p-1}\f{(1-n^{-1})_k^n}{(1)_k^n}\sum_{\substack{1\leq j<\langle n^{-1}\rangle_{p^2}\\(j,p)=1}}\f{1}{j}\eq \sum_{k=0}^{p-1}\f{(1-n^{-1})_k^n}{(1)_k^n}\l(\sum_{i=1}^k\f{1}{i}-\sum_{i=1}^k\f{1}{i-n^{-1}}\r)\pmod{p^2}.
\end{equation}
Clearly, we have
\begin{equation}\label{lefthand}
\begin{aligned}
\sum_{k=0}^{p-1}\f{(1-n^{-1})_k^n}{(1)_k^n}\sum_{\substack{1\leq j<\langle n^{-1}\rangle_{p^2}\\(j,p)=1}}\f{1}{j}\eq&\sum_{k=0}^{p-1}\binom{s+k}{s}^n\sum_{\substack{1\leq j<p^2-s\\(j,p)=1}}\f{1}{j}\\
\eq&\binom{s+k}{s}^n\l((m-s)\sum_{i=1}^m\f{1}{i^2}+\sum_{i=1}^m\f{1}{i}\r)\pmod{p^2}.
\end{aligned}
\end{equation}
by noting that $\sum_{k=1}^{p-1}k^{-1}\eq0\pmod{p^2}$ and $\sum_{k=1}^{p-1}k^{-2}\eq0\pmod{p}$.

Now we consider the right-hand side of \eqref{modp2}. For the sake of convenience, we let
$$
\Sigma_1:=\sum_{k=0}^{p-1}\f{(1-n^{-1})_k^n}{(1)_k^n}\sum_{i=1}^k\f{1}{i}
$$
and
$$
\Sigma_2:=\sum_{k=0}^{p-1}\f{(1-n^{-1})_k^n}{(1)_k^n}\sum_{i=1}^k\f{1}{i-n^{-1}}.
$$
We first evaluate $\Sigma_1$ modulo $p^2$. In light of \eqref{1mkn1knj20} we have
\begin{align*}
\Sigma_1\eq&\sum_{k=0}^{p-1}\binom{k+s}{s}^n\sum_{i=1}^k\f{1}{i}\eq\sum_{k=0}^{p-1-m}\binom{k+s}{s}^n\sum_{i=1}^{k}\f{1}{i}=\sum_{k=m}^{p-1}\binom{p-1-k+s}{s}^n\sum_{i=1}^{p-1-k}\f{1}{i}\\
\eq&\sum_{k=m}^{p-1}\binom{p-1-k+s}{s}^n\l(p\sum_{i=1}^k\f{1}{i^2}+\sum_{i=1}^k\f{1}{i}\r)\eq\sum_{k=m}^{p-1}\binom{p-1-k+s}{s}^n\sum_{i=1}^k\f{1}{i}\pmod{p^2}.
\end{align*}
Substituting $k$ for $k-m$ we obtain that
\begin{equation}\label{righthand1}
\begin{aligned}
\Sigma_1\eq&\sum_{k=0}^{p-1-m}\binom{p-1-k-m+s}{s}^n\sum_{i=1}^{k+m}\f{1}{i}\\
\eq&\sum_{k=0}^{p-1-m}\binom{k+m}{m}^n\Bigg(1-pn\l(\sum_{i=1}^{k+m}\f{1}{i}-\sum_{i=1}^k\f{1}{i}\r)\\
&+n(s-m)\l(\sum_{i=1}^k\f{1}{i}-\sum_{i=1}^m\f{1}{i}\r)\Bigg)\sum_{i=1}^{k+m}\f{1}{i}\pmod{p^2}.
\end{aligned}
\end{equation}
Below we consider $\Sigma_2$ modulo $p^2$. It is clear that
\begin{equation}\label{righthand2}
\begin{aligned}
\Sigma_2\eq&\sum_{k=0}^{p-1-m}\binom{k+s}{s}^n\sum_{i=1}^k\f{1}{i+s}=\sum_{k=0}^{p-1}\binom{k+s}{s}^n\l(\sum_{\substack{i=1\\(i,p)=1}}^{k+s}\f{1}{i}-\sum_{\substack{i=1\\(i,p)=1}}^{s}\f{1}{i}\r)\\
\eq&\sum_{k=0}^{p-1}\binom{k+s}{s}^n\sum_{\substack{i=1\\(i,p)=1}}^{k+s}\f{1}{i}-\sum_{k=0}^{p-1}\binom{k+s}{s}^n\l((m-s)\sum_{i=1}^m\f{1}{i^2}+\sum_{i=1}^m\f{1}{i}\r).
\end{aligned}
\end{equation}
With the help of \eqref{1mkn1knj20} we arrive at
\begin{align}\label{righthand3}
&\sum_{k=0}^{p-1}\binom{k+s}{s}^n\sum_{\substack{i=1\\(i,p)=1}}^{k+s}\f{1}{i}\eq\sum_{k=0}^{p-1-m}\binom{k+s}{s}^n\sum_{i=s-m+1}^{k+s}\f{1}{i}=\sum_{k=0}^{p-1-m}\binom{k+s}{s}^n\sum_{i=1}^{k+m}\f{1}{s-m+i}\notag\\
\eq&\sum_{k=0}^{p-1-m}\binom{k+s}{s}^n\l(-\sum_{i=1}^{k+m}\f{s-m-i}{i^2}\r)\eq\sum_{k=0}^{p-1-m}\binom{k+s}{s}^n\sum_{i=1}^{k+m}\f{1}{i}\notag\\
\eq&\sum_{k=0}^{p-1-m}\binom{k+m}{k}^n\l(1+n(s-m)\l(\sum_{i=1}^{k+m}\f{1}{i}-\sum_{i=1}^m\f{1}{i}\r)\r)\sum_{i=1}^{k+m}\f{1}{i}\pmod{p^2}
\end{align}
Substituting \eqref{lefthand}, \eqref{righthand1}, \eqref{righthand2} and \eqref{righthand3} into \eqref{modp2} we find that it suffice to show that
\begin{equation}\label{modp1}
\sum_{k=0}^{p-1-m}\binom{m+k}{k}^n\l(\sum_{i=1}^{k+m}\f{1}{i}-\sum_{i=1}^k\f{1}{i}\r)\sum_{i=1}^{k+m}\f{1}{i}\eq0\pmod{p}.
\end{equation}
In Section 2, we have defined $\Psi(x,y)$. By \eqref{lastsection1} and in view of \eqref{harcon1}, \eqref{harcon2}, \eqref{con3} and \eqref{1mkn1knj20} we obtain that
\begin{align}\label{final1}
&\f{\partial^2\Psi(x,y)}{\partial x^2}\bigg|_{(x,y)=(0,0)}\notag\\\eq&\sum_{k=0}^{p-1}\f{(1+m)_k^n}{(1)_k^n}\l(2\l(\sum_{i+1}^{k+m}\f{1}{i}\r)^2-2\sum_{i=1}^{k+m}\f{1}{i}\sum_{i=1}^k\f{1}{i}+\l(\sum_{i=1}^{m}\f{1}{i}\r)^2+\sum_{i=1}^{m}\f{1}{i^2}\r)\pmod{p}.
\end{align}
On the other hand, by \eqref{lastsection2} we find that
\begin{equation}\label{final2}
\f{\partial^2\Psi(x,y)}{\partial x^2}\bigg|_{(x,y)=(0,0)}\eq\sum_{k=0}^{p-1}\f{(1+m)_k^n}{(1)_k^n}\l(\sum_{i=1}^m\f{1}{i^2}+\l(\sum_{i=1}^m\f{1}{i}\r)^2\r)\pmod{p}.
\end{equation}
Combining \eqref{final1} and \eqref{final2} we immediately deduce \eqref{modp1}. The proof of Theorem \ref{main2} is now complete. \qed

\end{document}